\documentclass[11pt,colorinlistoftodos]{article}

\usepackage{blindtext}
\usepackage{geometry}

\usepackage[utf8]{inputenc} 
\usepackage[T1]{fontenc}    
\usepackage[hidelinks]{hyperref}
\usepackage{url}            
\usepackage{booktabs}       
\usepackage{amsfonts}       
\usepackage{amsmath}
\usepackage{amsthm}
\usepackage{amssymb}
\usepackage{nicefrac}       
\usepackage{microtype}      
\usepackage{cleveref}       
\usepackage{lipsum}         
\usepackage{graphicx}
\usepackage{doi}
\usepackage{color}
\usepackage{tikz-cd}
\usepackage{lmodern}
\usepackage{mathtools}
\usepackage[thinc]{esdiff}
\usepackage{cases}
\usepackage{empheq}

\usepackage{cleveref}
\usepackage{lipsum} 
\usepackage{todonotes}
\usepackage{orcidlink}
\usepackage{xcolor}
\usepackage{authblk}

\newcommand{\C}{\mathbb{C}}

\newcommand{\Z}{\mathbb{Z}}
\newcommand{\F}{\mathbb{F}}

\newtheorem{thm}{Theorem}[section]

\newtheorem{corollary}[thm]{Corollary}
\newtheorem{prop}[thm]{Proposition}

\newtheorem*{thm*}{Theorem} 

\theoremstyle{definition}
\newtheorem{definition}[thm]{Definition}
\newtheorem{remark}{Remark}[section]
\newtheorem{ex}{Example}[thm]

\DeclareMathOperator{\End}{End}
\DeclareMathOperator{\Der}{Der}
\DeclareMathOperator{\IDer}{IDer}

\DeclareMathOperator{\id}{id}
\DeclareMathOperator{\Zz}{Z}

\usepackage[%
backend=bibtex,%
style=numeric-comp,%
sorting=nyt,%
hyperref,%
]{biblatex}
\title{A New Definition of Superbiderivations for Lie Superalgebras}

\author[1,2]{Alfonso Di Bartolo}
\author[4,5]{Francesco Paolo Di Fatta}
\author[1,3]{Gianmarco La Rosa}

\affil[1]{Dipartimento di Matematica e Informatica, \protect\\ Università degli Studi di Palermo,\protect\\ Via Archirafi 34, 90123 Palermo, Italy}
\affil[2]{{\href{mailto:alfonso.dibartolo@unipa.it}{\texttt{alfonso.dibartolo@unipa.it}}}}
\affil[3]{{\href{mailto:gianmarco.larosa@unipa.it}{\texttt{gianmarco.larosa@unipa.it}}}}
\affil[4]{Doctoral School on Mathematics and Computational Sciences,\protect\\
University of Messina, Catania and Palermo}
\affil[5]{{\href{mailto:francesco.difatta@studenti.unime.it
}{\texttt{francesco.difatta@studenti.unime.it
}}}}

\date{%
}

\setuptodonotes{color=blue!30}

\definecolor{math}{rgb}{0.0, 0.8, 0.6}

\begin{document}
\maketitle

\begin{abstract}
In this paper, we study superbiderivations on Lie superalgebras from structural and geometric perspectives. Motivated by the classical fact that the bracket of a Lie algebra is itself a biderivation, we propose a new definition of superbiderivation for Lie superalgebras—one that requires the bracket to be a superbiderivation, a condition not satisfied by existing definitions in the literature.
Our focus is on complete Lie superalgebras, a natural generalization of semisimple Lie algebras that has emerged as a promising framework in the search for alternative structural notions. In this setting, we introduce and study linear supercommuting maps, comparing our definition with previous proposals.
Finally, we present two applications: one involving the superalgebra of superderivations of the Heisenberg Lie superalgebra and another offering initial geometric insights into deformation theory via superbiderivations.
\end{abstract}


\textbf{Keywords:} Superbiderivations, Complete Lie superalgebras,  Linear super commuting
 maps

\textbf{MSC 2020:} 17B40, 17B65, 17B66

\section{Introduction}\label{sec:introduction}

The notion of superalgebra has emerged as a natural framework in various disciplines, particularly within the domains of Mathematics and Geometrical Physics. A Lie superalgebra is a \( \mathbb{Z}_2 \)-graded generalisation of a Lie algebra, where the algebraic structure is defined by a bracket obeying graded versions of antisymmetry and the Jacobi identity \cite{KAC19778,Scheunert1979}.

Exploring biderivations is a conventional line of inquiry in the theory of Lie algebras (see, among others, \cite{wangyu2013,hanwangxia2016,chen2016,changchen2019,changchenzhou2019,changchenzhou2021}), which are bilinear maps, acting as derivations in each argument. It so happens that classical problems associated with Lie algebras have been developed in the context of Lie superalgebras, as well as the notion of a biderivation into the superalgebraic context appears to naturally extend, and has been studied. The first appearance of the notion of a biderivation on a Lie superalgebra (called \emph{superbiderivation}) appears to be in \cite{xu2015properties}, where Xu and Wang introduced and studied such maps in 2015 in the context of the Heisenberg superalgebra— a Lie superalgebra analogue of the classical Heisenberg Lie algebra. Subsequently, a more classical approach to the study of superbiderivations on Lie superalgebras was developed in \cite{fan2017superbiderivations}, where Fan and Dai focus on the centerless super-Virasoro algebra. It has been demonstrated that all superbiderivations in this case are inner, and the structure of linear super-commuting maps has been investigated. Following these two works, several notable studies appeared in the literature, where (super)biderivations were investigated in broader and diverse contexts (\cite{Xia01122016,Bai2023, Cheraghpour03042019,wangyu2013,hanwangxia2016,chen2016,changchen2019,changchenzhou2019,changchenzhou2021}). 

Setting aside the topic of biderivations for the moment, one of the most active and enduring areas of research in Lie algebras is their classification, more precisely the classification of nilpotent Lie algebras. Over the years, various approaches have been proposed: assumptions about the dimension of the center, the dimension of the derived subalgebra, the nilpotency index, and so on (\cite{bartolone2011,bartolone2018solvable, biggs2017class,KHUHIRUN2015,JOHARI2021}).
This variety of approaches stems from the fact that the Levi decomposition holds for Lie algebras: every Lie algebra can be decomposed as a semidirect product of a semisimple part (the Levi subalgebra) and a solvable ideal. All simple, and thus semisimple, Lie algebras have been fully classified over $\mathbb{C}$ and $\mathbb{R}$ (\cite{erdmann2006introduction, humphreys2012introduction}); however, the classification of solvable Lie algebras remains an open problem.
Since every solvable Lie algebra can be constructed from a nilpotent ideal (the nilradical) and its derivations (\cite{malcev1945}), classifying solvable Lie algebras largely amounts to classifying nilpotent Lie algebras. Unfortunately, this is well known to be computationally \emph{wild}.

In light of this, several alternative approaches to the classical classification of nilpotent Lie algebras have been developed over the years. One of the most promising approaches, which avoids the computational costs typically associated with such classifications, uses the geometric theory of deformations to classify $k$-step nilpotent Lie algebras, as discussed in \cite{Goze2015}. Nevertheless, the connection between Lie superalgebras, deformation, and cohomology theory is much deeper than we can cover in this paper. We refer interested readers to references \cite{Corwin1975} and \cite{Milnor1965} for more information.

In this context, however, another line of research has progressively emerged: are there alternative definitions of semisimple Lie algebras? Over time, several candidates have been proposed, such as \emph{complete} Lie algebras (those with all derivations inner and trivial center), \emph{sympathetic} Lie algebras (\emph{perfect} Lie algebras, i.e., satisfying $[L,L]=L$, and complete), among others.
In particular, for the latter class, Angelopoulos constructed in \cite{angelopoulos} a family of such Lie algebras that are not semisimple. 
The situation is more intricate for Lie superalgebras. In fact, Levi's decomposition theorem generally does not hold for Lie superalgebras (see \cite{KAC19778}).

However, these two structures and approaches have one thing in common: the study of biderivations, particularly in the simple, semisimple, and complete cases (\cite{DIBARTOLOLAROSA_COMPLETE_BIDER,tang2018}). It is well known that derivations of simple (and semisimple) Lie algebras are inner, and the same holds for Lie superalgebras.
The definitions of biderivations for Lie superalgebras vary across the literature. Among the numerous studies on superbiderivations, \cite{Yuan2018} is particularly noteworthy, as it investigates superbiderivations on classical simple Lie superalgebras.

In our work, we take a different approach, revisiting the definition of a biderivation itself. Specifically, just as in the case of Lie algebras, we require that the bracket of a Lie superalgebra be a superbiderivation. This property is not satisfied by the existing definitions in the literature.
Furthermore, we study these superbiderivations in the context of complete Lie superalgebras. Consequently, we introduce the notion of linear supercommuting maps and investigate their properties in the context of complete Lie superalgebras.

After this introduction, \Cref{sec:preliminaries_and_basic:idea} contains a collection of definitions and preliminary results necessary for understanding this paper, as well as the notational conventions we will adopt.
\Cref{sec:main results} presents the main results discussed above.
\Cref{sec:linearcomm} studies linear supercommuting maps, with a focus on comparing the existing definitions in the literature with the one introduced here.
\Cref{sec:applications} contains two applications: one concerning the superalgebra of superderivations of the complete Heisenberg Lie superalgebra and another of a more geometric nature. In the latter, we lay the groundwork for a perspective based on deformation theory.

\section{Preliminaries and basic idea}\label{sec:preliminaries_and_basic:idea}

We recall here that a \emph{superalgebra} is a $\mathbb{Z}_2$-graded vector space $L=L_0\oplus L_1$ over a field $\mathbb{F}$, endowed with an algebraic structure such that $L_\alpha L_\beta \subseteq L_{\alpha + \beta}$, for any $\alpha, \beta \in \mathbb{Z}_2$.
     A \emph{Lie superalgebra} is a superalgebra associated with a bilinear map                       $[\cdot,\cdot]\colon L\times L\rightarrow L$ satisfying the following properties:

\begin{align}
    &[L_{\alpha}, L_{\beta}]\subseteq L_{\alpha + \beta},\,\text{ for any }\alpha,\beta \in \mathbb{Z}_2\label{eq:z2_gradation},\\  
    &[x, y] = -(-1)^{\vert x \vert \vert y \vert} [y, x]\label{eq:super_skewsimmetry},\\ &(-1)^{\vert x \vert \vert z \vert}[x,[y,z]] + (-1)^{\vert y \vert \vert x \vert}[y,[z,x]] + (-1)^{\vert z \vert \vert y \vert}[z,[x,y]] = 0\label{eq:graded_jacobi_id},
\end{align}
 for all $x,y,z$ homogeneous elements in $L$.

By convention, the elements in a (Lie) superalgebra $L$ that has degree $0$ are said to be \emph{even}, and that ones with degree $1$ are said to be \emph{odd}. We will use this convention throughout this paper.
In addition, unless stated otherwise, the Lie superalgebra is assumed to be over a field of characteristic that is distinct from $2$. The term $L$ denotes a Lie superalgebra over a field  $\mathbb{F}$. Moreover, when in a formula, the degree of an element appears, it is understood that the element is homogeneous, i.e. it belongs to $L_\alpha$, for some $\alpha\in\mathbb{Z}_2$.

\begin{definition}\label{def:linear_homo_map}
A linear map \(f\colon L \mapsto L\) is said \emph{homogeneous of degree} \(\vert f\vert \in \big\{0,1\big\}\) if \(f(L_i) \subseteq L_{i + \vert f \vert}\), with \(i=0,1\).
\end{definition}

\begin{definition}\label{def:superderivation}

A homogeneous map \(D\colon L \mapsto L\) with degree \(\vert D \vert \) is said \emph{superderivation} if

\begin{equation}\label{eq:superderivation}
    D([x,y])=[D(x),y]+(-1)^{\vert D \vert \vert x \vert}[x,D(y)],
\end{equation}
for every $x,y\in L$.
\end{definition}

The set of all superderivations is known as the \emph{superderivation algebra} of $L$, which is denoted by $\Der(L)$. The algebra $\Der(L)$ has a natural $\mathbb{Z}_2$-gradation, i.e., 
\begin{equation*}
    \Der(L)=\Der_0(L)\oplus \Der_1(L),
\end{equation*}
 
where $\Der_0(L)$ is the set of all even superderivations and $\Der_1(L)$ is the set of all odd superderivations. Moreover, the superalgebra of superderivations is a subsuperalgebra of $\End(L)$ with the usual graded commutator (or supercommutator) given by

\begin{equation}\label{eq:supercomm_Der}
    [D_1,D_2]=D_1 D_2 - (-1)^{\vert D_1 \vert \vert D_2 \vert }D_2 D_1.
\end{equation}

It is easy to check that \Cref{eq:graded_jacobi_id} it is equivalent to
\begin{equation*}
    [x,[y,z]]=[[x,y],z]+(-1)^{\vert x \vert \vert y \vert}[y,[x,z]].
\end{equation*}
By the latter equation, if \(x\) is an homogeneus element of \(L\), it follows the map \(L_x:L \mapsto L\), with \(L_x(y)=[x,y]\) is a superderivation of degree \(\vert x \vert\). The map \(L_x\) is called \emph{inner superderivation}. The set of inner superderivations is denoted by \(\IDer(L)\), and it can be proved that is an ideal of \(\Der(L)\). Also \(\IDer(L)\) has a natural $\mathbb{Z}_2$-gradation, i.e., 
\begin{equation*}
    \IDer(L)=\IDer_0(L)\oplus \IDer_1(L),
\end{equation*}

where \(\IDer_i(L)=\big\{L_x \, / x \in L_i\big\}\), \(i=0,1\).
\begin{definition}\label{def:complete_Lie_superalgebra}
    A Lie superalgebra $L$ is called a \emph{complete} Lie superalgebra if $L$ satisfies the following two conditions:
    \begin{enumerate}
        \item $\Zz(L)=\{0\}.$
        \item  $\Der(L)=\IDer(L).$
    \end{enumerate}
\end{definition}

\begin{remark}
The second condition is equivalent to 
$ \Der_i(L)=\IDer_i(L)$ for every \(i=0,1\). Indeed, if \(\Der_i(L)=\IDer_i(L)\) then trivially  $\Der(L)=\IDer(L)$. Conversely, is evident that \(\IDer_i(L) \subseteq \Der_i(L)\). Let \(D \in \Der_0(L)\), in particular \(D \in \Der(L)=\IDer(L)=\IDer_0(L)\oplus \IDer_1(L)\), then there exist \(x_0 \in L_0, \, x_1 \in L_1\) such that \(D=L_{x_0} + L_{x_1}\). For every \(y \in L_i\), \([x_0+x_1,y]=D(y) \subseteq L_i\). But \([x_1,y] \in L_{i+1},\) so \([x_1,y]=0_L\). As conseguence \(x_1 \in Z(L)\), i.e. 
\(x_1=0_L\). Hence \(D=L_{x_0} \in \IDer_0(L)\). Therefore \(\Der_0(L)=\IDer_0(L)\). Similary \(\Der_1(L)=\IDer_1(L)\)

\end{remark}

Now, we proceed to introduce a new family of homogeneous maps on $L$. So let us consider an homogeneous linear map $D\colon L\to L$ such that, for all $x,y\in L$,
\begin{equation}\label{eq:homomap_type_2}
    D([x,y])=(-1)^{\vert D \vert \vert y \vert}[D(x),y]+[x,D(y)].
\end{equation}

In order to avoid unnecessary complexity in the ensuing argument, we define two vector spaces, $\mathcal{F}_1$ and $\mathcal{F}_2$, as follows:

\begin{align}
    \mathcal{F}_1(L)&=\left\{D\colon L\to L\mid \text{\Cref{eq:superderivation}} \text{ holds},\text{ for all } x,y\in L\right\},\\
    \mathcal{F}_2(L)&=\left\{D\colon L\to L\mid \text{\Cref{eq:homomap_type_2}} \text{ holds},\text{ for all } x,y\in L\right\}.
\end{align}

If $D\in\mathcal{F}_1(L)$ (respectively $D\in\mathcal{F}_2(L)$), we say that $D$ is of \emph{type 1} (resp. \emph{type 2}). Clearly, the type 1 maps are exactly the superderivations. The vector space $\mathcal{F}_2(L)$ has a natural $\mathbb{Z}_2$-gradation, i.e., 
\begin{equation*}
    \mathcal{F}_2(L)=\mathcal{F}_2(L)_0\oplus\mathcal{F}_2(L)_1,
\end{equation*}
where $\mathcal{F}_2(L)_0$ is the set of all even type 2 maps and $\mathcal{F}_2(L)_1$ is the set of all odd type 2 maps. In a manner analogous to that observed in the case of left multiplications, by the graded Jacobi identity, the map \(R_x\colon L \to L\), with \(R_x(y)=[y,x]\), is a type 2 map, for every $x\in L$. These kinds of maps are called type 2 \emph{inner}.
As with homogeneous linear maps, the definition of bilinear maps is recalled here.

\begin{definition}
A bilinear map \(B\colon L \times L \to L\) is said \emph{homogeneous of degree} \(\vert B \vert \in \big\{0,1\big\}\), if
\(B(L_{\alpha},L_{\beta}) \subseteq L_{\alpha + \beta + \vert B \vert}\).
\end{definition}

\begin{definition}\label{defn:superbiderivation_1}
An homogeneous bilinear map \(B\colon L \times L \mapsto L\) is said \emph{superbiderivation} if, for all \(x\in L\), \(B(x,-)\) is a superderivation and \(B(-,x)\) is a type 2 map. 
\end{definition}

The previous definition is equivalent to the following one.

\begin{definition}\label{defn:superbiderivation_2}
    An homogeneous bilinear map \(B\colon L \times L \to L\) is said \emph{superbiderivation} of $L$ if, for all \(x,y,z\in L\), the following equations hold
    \begin{align}
        B(x,[y,z])&=[B(x,y),z]+(-1)^{(\vert B \vert + \vert x \vert) \vert y \vert}[y,B(x,z)],\label{eq:superbiderivation_eq_1}\\
        B([x,y],z)&=(-1)^{(\vert B \vert + \vert z \vert) \vert y \vert}[B(x,z),y]+[x,B(y,z)].\label{eq:superbiderivation_eq_2}
    \end{align}
\end{definition}

\begin{ex}
    A basic example of a superbiderivation in a Lie superalgebra is its Lie superbracket. Indeed, since  $[x,-]=L_x$ and $[-,x]=R_x$, by \Cref{defn:superbiderivation_1} it follows that $[\cdot,\cdot]$ is a superbiderivation.
\end{ex}

Among the definitions of superbiderivations of a Lie superalgebra found in the literature, the one given by J. Yuan and X. Tang in \cite{Yuan2018} in 2018 initially appears similar to the one presented in this work. 

\begin{definition}[c.f. \cite{Yuan2018}]\label{def Tang}
   An homogeneous bilinear map \(B\colon L \times L \to L\) is said \emph{superbiderivation} of $L$ if, for all homogeneous \(x,y,z\in L\), the following equations hold
    \begin{align}
        B(x,[y,z])&=[B(x,y),z]+(-1)^{(\vert B \vert + \vert x \vert) \vert y \vert}[y,B(x,z)],\label{eq:superbiderivation_eq_1_Tang}\\
        B([x,y],z)&=(-1)^{\vert y \vert  \vert z \vert}[B(x,z),y]+(-1)^{\vert x \vert  \vert B \vert}[x,B(y,z)].\label{eq:superbiderivation_eq_2_Tang}
    \end{align}  
\end{definition}

However, the two definitions are not equivalent, as the following example shows, i.e. there exist homogeneous bilinear maps on L which satisfy one definition but not the other and vice versa (\Cref{ex:different_defns1}).

\begin{ex}\label{ex:different_defns1}
    Clearly, for a bilinear map with \(\vert B \vert=0\) the \Cref{defn:superbiderivation_2,def Tang} coincide. However, if \(\vert B \vert=1\), the definitions can be different. To show this, let \(L=L_0 \oplus L_1\), with \(L_0=\langle x \rangle\), \(L_1=\langle v \rangle\), and the only non-zero bracket is \([v,v]=x\). Since \([L,L]\) lies in the center of $L$, $L$ is a Lie superalgebra because the graded Jacobi identity is trivially satisfied. Let \(B \colon L \times L \to L\) be a map of degree $1$. Then

\begin{center}
    \(B(x,x)=\alpha v, \, \, \, \, B(x,v)=\beta x\),\,\,\,\,
    \(B(v,x)=\gamma x, \, \,\text{and } \, \, B(v,v)=\delta v\),
\end{center}
for some $\alpha,\beta,\gamma,\delta\in\F$. We now impose \(B\) to be a superbiderivation of degree $1$ according to \Cref{defn:superbiderivation_2}. To do this, we must impose that \Cref{eq:superbiderivation_eq_1,eq:superbiderivation_eq_2} hold for the elements of the basis of $L$. We first show that not all choices of $x,v$ that run into $B$ give us non--trivial relations on scalars. Let us start with observing that on one hand $B(x,[x,v])=0$ and, on the other hand, we have
\begin{equation*}
    [B(x,x),v]+(-1)^{(1 + \vert x \vert ) \vert x \vert }[x,B(x,v)]=\alpha x.
\end{equation*}
This implies that $\alpha=0$, that is $B(x,x)=0$. So $B(x,u)\in L_0=\Zz(L)$ (same for $B(u,x)$), for every $u\in L$. Then, $B(x,[u_1,u_2])=0$, for any $u_1,u_2\in L$, because $[u_1,u_2]\in L_0$. Hence $[B(x,u_1),u_2]=[u_1,B(x,u_2)]=0$. Moreover, for any $u\in L$, $[x,u]=0$, and then $B(v,[x,u])=0$. In addition, $B(v,x)\in L_0$, and it follows that $[B(v,x),u]=0$. Hence the following holds
\begin{equation*}
    B(v,[x,u])=[B(v,x),u]+(-1)^{(1 + \vert v \vert ) \vert x \vert }[x,B(v,u)]=0.
\end{equation*}
Similar arguments imply that 
\begin{equation*}
    B(v,[u,x])=[B(v,u),x]+(-1)^{(1 + \vert v \vert ) \vert u \vert }[u,B(v,x)]=0.
\end{equation*}
So it remains to check \Cref{eq:superbiderivation_eq_1} when the triple $\left\{v,v,v\right\}$ is considered. By the assumptions, $B(v,[v,v])=\gamma x$. But, at the same time, we must have
\begin{equation*}
    [B(v,v),v]+(-1)^{(1 + \vert v \vert ) \vert v \vert }[v,B(v,v)]=2 \delta x.
\end{equation*}
This implies the condition $\gamma=2\delta$. 
To verify if \Cref{eq:superbiderivation_eq_2} holds, one can apply the same arguments to check this equation only for the triple $\left\{v,v,v\right\}$. Since $B([v,v],v)=\beta x$, and 
\begin{equation*}
    (-1)^{(1 + \vert v \vert) \vert v \vert}[B(v,v),v]+[v,B(v,v)]=2 \delta x,
\end{equation*}
\Cref{eq:superbiderivation_eq_2} holds if and only if $\beta=2\delta$. 
Hence, a superbiderivation of degree $1$ is necessarily of the form:

\begin{center}
    \(B(x,x)=0, \, \, \, \, B(x,v)=B(v,x)=2\delta x,\,\,\,\text{and } \, \, B(v,v)=\delta v\).
\end{center}

By similar reasoning, such a bilinear map satisfies conditions \Cref{eq:superbiderivation_eq_1_Tang} and \Cref{eq:superbiderivation_eq_2_Tang} if and only if it is of the form
\begin{equation*}
    B(x,x)=0, \, \, \, \, B(x,v)=-2\delta x,\,\,\,\,B(v,x)=2\delta x,\,\,\,\,\text{and } \, \, B(v,v)=\delta v.
\end{equation*}
Finally, a non-zero bilinear map $B$ of degree $1$ cannot be a superbiderivation under both definitions simultaneously.
\end{ex}

\section{Main results}\label{sec:main results}

We are now ready to present a set of results concerning the relationship established by maps of type 1 and type 2. In fact, we can obtain a map of type 2 from a map of type 1 with a multiplication by a special factor and vice versa.  

\begin{prop}\label{prop:antiderivazione}
Let \(D\) be a superderivation of \(L\) of degree \(\vert D \vert\). Let \(\delta\colon L\to L\) be a map such that:

\begin{enumerate}
    \item \(\delta(x)=(-1)^{\vert D \vert \vert x \vert} D(x)\),  for $x\in L$.

    \item \(\delta(x_0 + x_1)=\delta(x_0)+\delta(x_1)\), for all $x_0 \in L_0$, $x_1 \in L_1$.
\end{enumerate}
Then \(\delta\in\mathcal{F}_2(L)\).
\end{prop}

\begin{proof}
If \(x,y \in L\) have the same degree, then \( \vert x+y \vert = \vert x\vert =\vert y \vert\). Since \(D\) is linear, by the two conditions, it follows that \(\delta\) is linear.
Moreover, \(\delta\) is also homogeneous, with \(\vert \delta \vert =\vert D \vert\) and, since $D$ is a superderivation, for every $x,y\in L$, at last we have

\begin{align*}
  \delta([x,y]) &= (-1)^{\vert D \vert \vert [x,y] \vert }D([x,y]) &\\
    &= (-1)^{\vert D \vert ( \vert x \vert + \vert y \vert)} \big\{[D(x),y]+(-1)^{\vert D \vert \vert x \vert}[x,D(y)] \big\} &\\
    &= (-1)^{\vert D \vert ( \vert x \vert + \vert y \vert)} [D(x),y] + (-1)^{\vert D \vert (\vert x \vert + \vert y \vert)+ \vert D \vert \vert x \vert} [x,D(y)]  &\\
    & =(-1)^{\vert D \vert \vert y \vert} [(-1)^{\vert D \vert \vert x \vert}D(x),y] + [x,(-1)^{\vert D \vert \vert y \vert} D(y)] &\\
    & = (-1)^{\vert \delta \vert \vert y \vert} [\delta(x),y] + [x,\delta(y)].
\end{align*}
    
\end{proof}

With similar arguments, one can prove the result below. 

\begin{prop}\label{prop:derivazioni}
Let $\delta\in\mathcal{F}_2(L)$ of degree \(\vert \delta \vert\). Let \(D\colon L\to L\) be a map such that:

\begin{enumerate}
    \item \(D(x)=(-1)^{\vert \delta \vert \vert x \vert} \delta(x)\), for $x\in L$.

    \item \(D(x_0 + x_1)=D(x_0)+D(x_1)\), for all $x_0 \in L_0$, $x_1 \in L_1$.
\end{enumerate}
Then \(D\) is a superderivation of $L$.
\end{prop}

The set of homogeneous elements is a set of generators for \(L\). In some sense, we cannot consider the second property of \Cref{prop:antiderivazione} and \Cref{prop:derivazioni} if we impose the linearity. So we have two maps $f\colon\Der(L)\to\mathcal{F}_2(L)$ and $g\colon\mathcal{F}_2(L)\to\Der(L)$ such that

\begin{equation*}
f(D)(x)=(-1)^{\vert D \vert \vert x \vert}D(x) \quad\text{ and }\quad g(\delta)(x)=(-1)^{\vert \delta \vert \vert x \vert}\delta(x),
\end{equation*}

for all $x\in L$.
The maps $f$ and $g$ are (linear) homogeneous of degree $0$ (they do not change the degree). 
Indeed, for every $D_1,D_2\in\Der_\alpha(L)$, with respect to the map $f$, we have
\begin{align*}
    f(D_1+D_2)(x)&=(-1)^{\vert D_1+D_2\vert\vert x\vert}\left(D_1+D_2\right)(x)\\
    &=(-1)^{\vert D_1 \vert \vert x \vert}D_1(x)+(-1)^{\vert D_2 \vert \vert x \vert}D_2(x)\\
    &=f(D_1)(x)+f(D_2)(x),
\end{align*}
for all $x\in L$. The latter is true because $D_1$ and $D_2$ have the same degree. Therefore, $f$ is linear because one can extend this argument to every superderivation, since such a map can be written as a sum of a superderivation of degree $0$ and a superderivation of degree $1$.
Furthermore, for all $x\in L$, we have
\begin{equation*}
    f(g(\delta))(x) = (-1)^{\vert g(\delta) \vert \vert x \vert} g(\delta)(x).
\end{equation*}

Since $g$ (and also $f$) has degree $0$, we can conclude that

\begin{equation*}
    f(g(\delta))(x)=(-1)^{\vert \delta \vert \vert x \vert} g(\delta)(x) = (-1)^{\vert \delta \vert \vert x \vert} (-1)^{\vert \delta \vert \vert x \vert} \delta(x).
\end{equation*}

That is \(fg=\id_{\mathcal{F}_2(L)}\) and, similary,
\(gf=\id_{\Der(L)}\). Thus, the results outlined above can be summarised as follows.

\begin{thm}\label{thm:anti/der_iso}
The vector space \(\mathcal{F}_2(L)\) is isomorphic to \(\Der(L)\).   
\end{thm}

It is clear that the space $\mathcal{F}_2(L)$ naturally inherits a Lie superalgebra structure via the standard supercommutator, as will be made precise in the following result.

\begin{prop}\label{prop:antider_is_Lie_superalgebra}
The space \(\mathcal{F}_2(L)\) with bracket given by
\([\delta_1,\delta_2]=\delta_1\delta_2 - (-1)^{\vert \delta_1 \vert \vert \delta_2 \vert} \delta_2 \delta_1\) is a Lie superalgebra.
\end{prop}

\begin{proof}
Obviously, \([\delta_1,\delta_2]\) is still a linear homogeneous map of $L$. We should only prove the closure of the bracket, i.\ e.\ that \([\delta_1,\delta_2]\) lies in $\mathcal{F}_2(L)$. 

Let \(D_1,D_2\) the associated superderivations, i.e. \(\delta_i(x)=(-1)^{\vert D_i \vert \vert x \vert}D_i(x)\) for all \(x\in L\), with \(i=1,2\). Clearly it holds \(\vert \delta_i \vert = \vert D_i \vert \). Then,

\begin{align*}
       [\delta_1,\delta_2](x)&=\delta_1\delta_2(x)-(-1)^{\vert \delta_1 \vert \vert \delta_2 \vert}\delta_2 \delta_1(x)\\
    &=\delta_1((-1)^{\vert D_2 \vert \vert x \vert}D_2(x))- (-1)^{\vert D_1 \vert \vert D_2 \vert}\delta_2((-1)^{\vert D_1 \vert \vert x \vert} D_1(x)) \\
    &=(-1)^{\vert D_2 \vert \vert x \vert}\delta_1((D_2(x))- (-1)^{\vert D_1 \vert \vert D_2 \vert+\vert D_1 \vert \vert x \vert}
    \delta_2(D_1(x)) \\
    &=(-1)^{\vert D_2 \vert \vert x \vert}(-1)^{\vert D_1 \vert \vert D_2(x) \vert} D_1 D_2(x) 
    - (-1)^{\vert D_1 \vert \vert D_2 \vert+\vert D_1 \vert \vert x \vert}(-1)^{\vert D_2 \vert \vert D_1(x) \vert} D_2 D_1(x) \\
    &=(-1)^{\vert D_2 \vert \vert x \vert + \vert D_1 \vert (\vert D_2 \vert + \vert x \vert)} D_1 D_2(x) - (-1)^{\vert D_1 \vert \vert D_2 \vert+\vert D_1 \vert \vert x \vert  + \vert D_2 \vert (\vert D_1 \vert + \vert x \vert)} D_2 D_1(x) \\
    &=(-1)^{(\vert D_1 \vert + \vert D_2 \vert) \vert x \vert + \vert D_1 \vert \vert D_2 \vert } D_1 D_2(x) - (-1)^{(\vert D_1 \vert + \vert D_2 \vert) \vert x \vert + 2 \vert D_1 \vert \vert D_2 \vert} D_2 D_1(x) \\
    &=(-1)^{\vert D_1 \vert \vert D_2 \vert}
    (-1)^{(\vert D_1 \vert + \vert D_2 \vert) \vert x \vert} \left( D_1 D_2(x) - (-1)^{\vert D_1 \vert \vert D_2 \vert} D_2D_1(x)\right) \\
    &=(-1)^{\vert D_1 \vert \vert D_2 \vert}
    (-1)^{\vert [D_1,D_2] \vert \vert x \vert}  [D_1,D_2](x), \\
\end{align*}

for any $x\in L$. Since \([D_1,D_2]\) is a superderivation, then \(\delta_3(x)=(-1)^{\vert [D_1,D_2] \vert \vert x \vert} [D_1,D_2](x)\) is a map of type 2, as it is  his multiple \((-1)^{\vert D_1 \vert \vert D_2 \vert} \delta_3\) and this proves the statement.
\end{proof}

A natural question that arises from the previous result is whether \( f \) and \( g \) are isomorphisms of Lie superalgebras. Unfortunately, the answer is negative, as the following remark illustrates.

\begin{remark}
    We observe that, when one computes $f([D_1,D_2])$, with $D_1,D_2\in\Der(L)$, by the proof one obtains $(-1)^{\vert D_1 \vert \vert D_2 \vert}[f(D_1),f(D_2)]$. This shows that, in general, $f$ (and also $g$) is not an isomorphism of Lie superalgebras.
\end{remark}

We are now ready to study the space \(\mathcal{F}_2(L)\) when \(L\) is a complete Lie superalgebra, as defined in \Cref{def:complete_Lie_superalgebra}. Under this assumption, we will analyse the behaviour of these maps. Furthermore, we will apply this result to the study of superbiderivations, following \Cref{defn:superbiderivation_2}.

\begin{prop}
Let \(L\) be a complete Lie superalgebra. Then every map in \(\mathcal{F}_2(L)\) is inner.   
\end{prop}

\begin{proof}
Let \(\delta \in \mathcal{F}_2(L)\) homogeneous. Then, by \Cref{thm:anti/der_iso} there exists a superderivation \(D\) such that 
\(\delta(x)=(-1)^{\vert D \vert \vert x \vert}D(x)\). Since \(L\) is complete, w.l.o.g. we know that there exists an element $t\in L$ such that \(D=L_t\). So \(\vert D \vert = \vert t \vert\), and

\begin{center}
    \(\delta(x)=(-1)^{\vert D \vert \vert x \vert}D(x)=(-1)^{\vert t \vert \vert x \vert}L_t(x)=
    (-1)^{\vert t \vert \vert x \vert}[t,x]=-[x,t]=[x,-t]\).   
\end{center}

The latter is equivalent to saying that \(\delta=R_{-t}\). Finally, if $\delta\in\mathcal{F}_2(L)$ is not homogeneous, then $\delta=\delta_0+\delta_1$, with $\delta_i\in\mathcal{F}_2(L)_i$, where $i=1,2$. Hence $\delta_i=R_{-t_i}$, for such a $t_i\in L_i$, and $\delta=R_{-t_0-t_1}$.
\end{proof}

\begin{thm}\label{thm:Teorema_complete}
Let \(L\) be a complete Lie superalgebra and \(B\) a homogeneous bilinear map of $L$ with degree \(\vert B \vert\). Then, $B$ is a superbiderivation of $L$ if and only if there exist two linear homogeneus maps \(\varphi,\psi\) of degree \(\vert B \vert\) such that

\begin{equation*}
    B(x,y)=[\varphi(x),y]=[x,\psi(y)],\,\,\text{ for all } x,y \in L.
\end{equation*}
\end{thm}

\begin{proof}
The "if" direction is the easiest to prove. Indeed, if there exist two homogeneous linear maps $\varphi,\psi\colon L\to L$ such that $B(x,y)=[\varphi(x),y]=[x,\psi(y)]$, then $B$ is a superbiderivation of $L$ (w.r.t. \Cref{defn:superbiderivation_2}) because $B(x,-)=L_{\varphi(x)}\in\Der(L)$ and $B(-,x)=R_{\psi(x)}\in\mathcal{F}_2(L)$.

Now we prove the "only if" direction. Let us consider a superbiderivation $B$ of $L$. Then, $B(x,-)$ (respectively $B(-,x)$) is a map of type $1$ (resp. type $2$) of $L$ of degree $ \vert B \vert + \vert x \vert$, for any $x\in L$. Since $L$ is complete, 
\(\mathcal{F}_1(L)\) (respectively \(\mathcal{F}_2(L)\)) has only left multiplications (resp. right multiplications). Thus, for every $x\in L$, there exist $u,v\in L$ such that $B(x,-)=L_u(-)=[u,-]$ and $B(-,x)=R_v(-)=[-,v]$. So we can define two maps $\varphi,\psi\colon L\to L$ in the following way:
\[
\varphi\colon x\mapsto u\quad\text{ and }\quad\psi\colon x\mapsto v.
\]
Now we have to prove that $\varphi$ and $\psi$ are homogeneous maps of degree $\vert B \vert $ .

\begin{itemize}
    \item Linearity: Let $x_1,x_2\in L$. We now consider the two elements $[\varphi(x_1+x_2),y]$ and $[\varphi(x_1)+\varphi(x_2),y]$, for any $y\in L$. By the definition of the map $\varphi$, and since $B$ is a bilinear map, we have $[\varphi(x_1+x_2),y]= [\varphi(x_1)+\varphi(x_2),y]$. This implies that $\varphi(x_1+x_2)-\varphi(x_1)-\varphi(x_2)\in\Zz(L)$. The center of $L$ is trivial because $L$ is complete, then $\varphi(x_1+x_2)=\varphi(x_1)+\varphi(x_2)$. In a similar way one can proves that \(\varphi(\lambda x)=\lambda \varphi(x) \), where $\lambda\in\F$. With the same arguments, $\psi$ is linear.

 \item Degree: Let \(\vert B \vert =0\) (the case 1 is analogous) and \(x \in L_0\). We have \(\varphi(x)=x_0 + x_1\), with \(x_i \in L_i\), \(i=0,1\). If \(y \in L_0\) then \([x_0,y]+[x_1,y]=[\varphi(x),y]=B(x,y)\in L_0\) since \(\vert B \vert=0\). \([x_0,y] \in L_0\) and \([x_1,y] \in L_1\), so \([x_1,y]=0\). With same arguments, if \(y \in L_1\), \([x_1,y]=0\); hence \(x_1\) lies in the center. So \(\varphi(x)=x_0 \in L_0\),  for all \(x \in L_0\), i.e \(\varphi(L_0) \subseteq L_0\). With similar computations \(\varphi(L_1) \subseteq L_1\). We can conclude that \(\varphi\) is an homogeneus map of degree \(0=\vert B \vert\). Similary also \(\psi\) has degree \(0=\vert B \vert\). The case \(\vert B\vert =1\) is analogous.
 
\end{itemize}

\end{proof}

\section{Linear supercommuting maps}\label{sec:linearcomm}

In this section, we propose an alternative definition of linear supercomputing maps for Lie superalgebras. Although such maps have already been defined in the literature, we will show that the existing definition is flawed. Consequently, a new definition is needed—one that aligns with the underlying structure of Lie superalgebras and, naturally, generalises the notion of linear supercomputing maps for Lie algebras, which can be seen as a special case of Lie superalgebras.

\subsection{Previous definitions of linear supercommuting maps}\label{subsec:issues}

Let $L$ be an arbitrary Lie superalgebra and let $f\colon L\to L$ be a linear map of $L$. In \cite{fan2017superbiderivations}, Fan and Dai say that $f$ is linear supercomputing on $L$ if  \([f(x),x]=0\), for every homogeneous $x\in L$.

\begin{remark}
    We want to emphasise that, if one would compute the bracket $[f(x+y),x+y]$ with a supercomputing map $f$ in order to obtain $0$, it is mandatory to consider homogeneous elements $x,y\in L$, with $\vert x \vert=\vert y \vert$. Were this not the case, we would have that $x+y$ is not homogeneous since $x\in L_\alpha$ and $y\in L_{\alpha+1}$.
\end{remark}

Let $x,y\in L$ homogeneous, with $\vert x \vert=\vert y \vert$. If a linear map $f$ satisfies the condition above, for such $x,y\in L$, then we have

\begin{align*}
[f(x+y),x+y]&=[f(x),x]+[f(x),y]+[f(y),x]+[f(y),y]\\
&=[f(x),y]+[f(y),x]\\
&=0.
\end{align*}

Hence, the following equations hold

\begin{equation*}
    [f(x),y]=-[f(y),x]=(-1)^{\vert f(y) \vert \vert x \vert} [x,f(y)].
\end{equation*}
 This last equation is clearly different from Equation (4.10) that appears in \cite{fan2017superbiderivations}. Moreover, even assuming that the map 
$f$ is homogeneous, we would still obtain

\begin{equation*}
    [f(x),y]=-[f(y),x]=(-1)^{\vert f(y) \vert \vert x \vert} [x,f(y)]=(-1)^{(\vert f \vert+\vert y \vert)\vert x \vert}[x,f(y)],
\end{equation*}
which nevertheless differs from Equation (4.10) in \cite{fan2017superbiderivations}. Now we are ready to show that, for example, Theorem 4.4 in \cite{fan2017superbiderivations} is not true, in general.  In particular, they consider the super-Virasoro algebra $S(s)$, with $s=0$ or $s=\frac{1}{2}$ with the even part spanned by \(\big\{L_i \, | \, i \in \Z\big\}\), the odd part spanned by \(\big\{G_k \, | \, k \in s+\Z\big\}\) and brackets

\begin{equation*}
    [L_i,L_j]=(j-i)L_{i+j} \qquad
    [L_i,G_j]=(j- \frac{i}{2})G_{i+j} \qquad
    [G_i,G_j]=2L_{i+j}
\end{equation*}

As a result, they conclude that all supercommuting maps of $S(s)$ have the form \(f(x)=\lambda x\), for every \(x\in L\) and \(\lambda\) a fixed element in $\F$. Unfortunately, this is not true.  Indeed, if we put $\lambda=1$ and we consider the bracket $[f(G_1),G_1]$, these imply that

\begin{equation*}
    [G_1,G_1]=2L_2\neq0.
\end{equation*}

One of the consequences of these outputs is that the statement Theorem 5.4 in \cite{Yuan2018} is also false. To show this, we consider the Lie superalgebra \(\mathfrak{sl}(1 \vert 2)\), with even part 

\begin{equation*}
    L_0=\left\{ \begin{pmatrix}
        a & 0 & 0\\
        0 & b & c \\
        0 & d & a-b
    \end{pmatrix} \, | \, a,b,c,d \in \C \right\}
\end{equation*}

and odd part

\begin{equation*}
    L_1=\left\{ \begin{pmatrix}
        0 & a & b\\
        c & 0 & 0 \\
        d & 0 & 0
    \end{pmatrix} \, | \, a,b,c,d \in \C \right\}
\end{equation*}

If we consider $\lambda=1$, and the element 

\begin{equation*}
    x=\begin{pmatrix}
         0 & 1 & 1 \\
        1 & 0 & 0 \\
        1 & 0 & 0 
    \end{pmatrix}
\end{equation*}

we have

\begin{align*}
    [f(x),x]&=[x,x]\\
    &=\begin{pmatrix}
        0 & 1 & 1 \\
        1 & 0 & 0 \\
        1 & 0 & 0
    \end{pmatrix}
    \begin{pmatrix}
        0 & 1 & 1 \\
        1 & 0 & 0 \\
        1 & 0 & 0
    \end{pmatrix}-(-1)^{\vert x \vert \vert x \vert}\begin{pmatrix}
        0 & 1 & 1 \\
        1 & 0 & 0 \\
        1 & 0 & 0
    \end{pmatrix}
    \begin{pmatrix}
        0 & 1 & 1 \\
        1 & 0 & 0 \\
        1 & 0 & 0
    \end{pmatrix}\\[.5em]
    &=\begin{pmatrix}
        2 & 0 & 0 \\
        0 & 1 & 1 \\
        0 & 1 & 1
    \end{pmatrix}+\begin{pmatrix}
        2 & 0 & 0 \\
        0 & 1 & 1 \\
        0 & 1 & 1
    \end{pmatrix}\\[.5em]
    &=\begin{pmatrix}
        4 & 0 & 0 \\
        0 & 2 & 2 \\
        0 & 2 & 2
    \end{pmatrix}\neq 0_{3\times 3}.
\end{align*}

\subsection{New definition of linear supercommuting maps}\label{subsec:new_definition_lincomm}

In literature, an important family of maps is the set of commuting maps. Let $A$ be an algebra over a field. A linear map $f\colon A\to A$ is called \emph{commuting map} if, for any $x,y\in A$,
    
    \begin{equation}\label{eq:defn_of_linear_comm_map1}
        f(x)y=xf(y).
    \end{equation}

For a Lie algebra $L$, if the field has not characteristic two, \Cref{eq:defn_of_linear_comm_map1} is equivalent to 

\begin{equation}
    [f(x),x]=0,
\end{equation}
for every $x\in L$.
As we say in \Cref{subsec:issues}, the conditions are not equivalent for a Lie superalgebra. Inspired by the original definition of an arbitrary algebra given by \Cref{eq:defn_of_linear_comm_map1}, the following definition is proposed.

\begin{definition}\label{defn:lin_comm_super}
    Let \(L\) be a Lie superalgebra. A linear map $f\colon L\to L$ is called \emph{linear supercommuting map} if 
    \begin{equation}\label{eq:defn_of_linear_comm_map2}
        [f(x),y]=[x,f(y)],
    \end{equation}
    for any $x,y\in L$.
\end{definition}

We immediately observe that the elements involved in the last definition are not necessarily homogeneous. This approach ensures that the family of linear maps under consideration is as general as possible.
Due to the superalgebraic structure of $L$, it naturally follows that $f$ can be decomposed into an even part and an odd part. We consider the projections \(\pi_i \colon L \to L_i\) in the subspace \(L_i\), with \(i=0,1\), that is $\pi_i(x)=\pi_i(x_0+x_1)=x_i$. Then, for every $x\in L$, we have
\begin{center}
    \(f(x)=\pi_0 f(x) + \pi_1 f(x)= \pi_0 f(\pi_0(x) + \pi_1(x)) + \pi_1 f(\pi_0(x) + \pi_1(x))=f_0(x)+f_1(x) \),
\end{center}

where \(f_0(x)=\pi_0 f \pi_0 (x) + \pi_1 f \pi_1 (x) \) and \(f_1(x)=\pi_1 f \pi_0 (x) + \pi_0 f \pi_1 (x)\). The maps \(f_0\) and \(f_1\) are also linear, and \(f_j(L_i) \subseteq L_{i+j}\), for all \(i,j\in\mathbb{Z}_2\). Hence, $f_0$ and $f_1$
  are homogeneous, and—as stated in the following proposition—they also supercommute.

\begin{prop}\label{prop:supercomm_decomposition}
Let \(f=f_0+f_1\) be a linear supercommuting map, with $f_0$ and $f_1$ written as above. Then $f_0$ and $f_1$ are linear supercommuting maps.
\end{prop}

\begin{proof}
Let \(x \in L_{\alpha}, y \in L_{\beta}\). Then
\([f(x),y]=[x,f(y)]\), i.e. 

\begin{equation}\label{eq:eq:supercomm_decomposition}
   [f_0(x),y]+[f_1(x),y]=[x,f_0(y)]+[x,f_1(y)].
\end{equation}

 The elements $[f_0(x),y]$ and $[f_1(x),y]$ are linearly independent since \([f_0(x),y] \in L_{\alpha + \beta}\) and \([f_1(x),y] \in L_{\alpha + \beta +1}\). Moreover, $[x,f_0(y)]$ and $[x,f_1(y)]$ are linearly indipendent. By \Cref{eq:eq:supercomm_decomposition}  follows that \([f_i(x),y]=[x,f_i(y)]\) for any homogeneous \(x,y \in L\), for \(i=0,1\). Since the homogeneous elements are a set of generators of \(L\), \([f_i(x),y]=[x,f_i(y)]\), for all \(x,y \in L\).
\end{proof}

In the next section, we will apply the results and definitions that have been established thus far to the study of linear supercomputing maps of the Lie superalgebra of derivations of the Heisenberg superalgebra.

\section{Applications }\label{sec:applications}

\subsection{The Lie superalgebra of superderivations of the Heisenberg superalgebra}

Let us start with the definition of Heisenberg superalgebra given by L.Y. Wang and D.J. Meng in \cite{Wang2003}. A Lie superalgebra \(L=L_0 \oplus L_1\) over the complex field $\mathbb{C}$ is said to be a \textit{Heisenberg superalgebra} if it satisfies the following conditions:

    \begin{itemize}
        \item \([L,L]=L_0=\Zz(L).\)
        \item  \(\dim L_0=1\).
    \end{itemize}

In this section, $L$ will denote the $n+1$-dimensional Heisenberg superalgebra, where $\dim L_1=n$. By these two conditions, let $\mathcal{B}=\left\{x_1,\ldots,x_n,c\right\}$ be a basis of $L$. If \(L_0=\langle c\rangle\) and \(L_1=\langle x_1,\ldots,x_n\rangle\), then there exists a 
 symmetric bilinear form \(\beta\colon L \times L \to \mathbb{F} \) such that \([x,y]=\beta(x,y) c\), for every \(x,y\) in \(L\). Since \(\beta\) is symmetric, we can choose \(x_i\in L_1\) such that \(\beta(x_i,x_j)=\delta_{ij}\), for every $1\leq j\leq n$. Let \(N\in M_{n+1}(\mathbb{C)}\) be the matrix associated to $\beta$ with respect to the basis $\mathcal{B}$, that is the matrix

\begin{equation*}
    N=\begin{pmatrix}
         I_n & 0_{n\times 1} \\
         0_{1\times n} & 0.
     \end{pmatrix}
\end{equation*}

We now aim to compute the derivation algebra of \(L\). 
Since we adopt a different basis from the one used in \cite{Wang2003} 
to prove Lemmas 1.2 and 1.3, we provide an alternative proof of the same result.

\begin{prop}
  Let \(D\) be a superderivation of the Heisenberg superalgebra L. Then, respect to the basis $\mathcal{B}$, the matrix $M$ associated to $D$ has the form

  \begin{equation*}
      M=\begin{pmatrix}
          \lambda I_n + B & 0 \\
          a & 2\lambda
      \end{pmatrix},
  \end{equation*}

where \(B\) is a \(n \times n\) skew-symmetric matrix, \(\lambda \in \mathbb{C}\), and \(a\in\mathbb{C}^n\equiv M_{1\times n}(\mathbb{C})\). 
\end{prop}
 
\begin{proof}
We will prove this result by considering different cases based on the degree of $D$.

\begin{itemize}
    \item If \(D\) has degree 0, then $M$ has the form

\begin{equation*}
    \begin{pmatrix}
            A & 0_{n\times 1} \\
            0_{1\times n} & \lambda
        \end{pmatrix}.
\end{equation*}

    On one hand, for any $x,y\in L$, we have \(D([x,y])=\beta(x,y)D(c)=\lambda\beta(x,y)c\). On the other hand, \([D(x),y]+[x,D(y)]=
    (\beta(D(x),y)+\beta(x,D(y)))c\). So we have

\begin{equation*}
    \lambda\beta(x,y)=\beta(D(x),y)+\beta(x,D(y)),
\end{equation*}
for every $x,y\in L$. This is equivalent to the matrix equation \(\lambda N= M^tN+NM\), i.e.

\begin{equation*}
    \begin{pmatrix}
                \lambda I_n & 0_{n\times1} \\
                0_{1\times n} & 0
                \end{pmatrix} = \begin{pmatrix}
                A^t+A & 0_{n\times1} \\
                0_{1\times n} & 0
            \end{pmatrix}.
\end{equation*}

    The last equation implies that \(A_{ij}=-A_{ji}\) for \(i \neq j\), \(A_{ii}=\frac{\lambda}{2}\) for every \(1\leq i\leq n\).

    \item If \(D\) has degree 1, then $M$ has the form

    \begin{equation*}
        \begin{pmatrix}
            0_n & b \\
            a & 0
        \end{pmatrix},
    \end{equation*}

   where $a\in M_{1\times n}(\mathbb{C})$ and $b\in M_{n\times 1}(\mathbb{C})$.  Since the subspace $[L,L]$ is $D$-invariant, i.e. \(D([L,L]) \subseteq [L,L]=\langle c\rangle\), thus \(D(c)=0_L\) and $b=0_{M_{n\times 1}(\mathbb{C})}$. The image subspace of \(D\) is always contained in the center of $L$, then the bracket $[D(x),c]=0_L$, for any $x\in L$, and then we do not have any constraints on $a$.
\end{itemize}   
\end{proof}

By Theorem 1.5 in \cite{Wang2003}, we know that the Lie superalgebra of superderivations of a Heisenberg superalgebra is complete. If we denote by \(e_{ij}\) the matrix with 1 in position \((i,j)\) and 0 otherwise, a basis of the superalgebra is given by \(\big\{A,B_{ij},C_i\big\}\), where
\(A=\sum_{k=1}^n e_{kk}+2e_{n+1,n+1}\),
\(B_{ij}=e_{ij}-e_{ji}\) for \(1\leq i<j\leq n\), and
\(C_i=e_{n+1,i}\) for \(1\leq i\leq n\). To simplify the computations, from now on, we consider \(B_{ij}=e_{ij}-e_{ji}\), for every $1\leq i,j\leq n$, with \(i\neq j\). The set \(\big\{A,B_{ij},C_i\big\}\) is still a set of generators and \(B_{ij}=-B_{ji}\). The elements \(A\) and \(B_{ij}\) have degree 0, while \(C_i\) have degree 1. In other words,

\begin{equation*}
    \Der_0(L)=\langle A, B_{ij} \rangle \quad\text{ and }\quad \Der_1(L)=\langle C_i \rangle.
\end{equation*}

The non-zero brackets of $\Der(L)$ are given by

\begin{align*}
&[A,C_i]=-[C_i,A]=C_i \\  
&[B_{ij},B_{kl}]=-[B_{kl},B_{ij}]=
    \delta_{jk}B_{il}+
    \delta_{il}B_{jk}
    +\delta_{jl}B_{ki}
    +\delta_{ki}B_{lj}\\
&[B_{ij},C_k]=-[C_k,B_{ij}]=
    \delta_{kj}C_i-\delta_{ki}C_j
\end{align*}

\begin{prop}\label{prop:VarPhigrade0}
    Let $n\geq3$ and let \(\varphi, \psi\colon \Der(L)\to\Der(L)\) be two homogeneous maps of degree $0$ such that
    
    \begin{equation}\label{eq:complex_num_homo0}
        [\varphi(x),y]=[x,\psi(y)],
    \end{equation}
    for every $x,y\in \Der(L)$. Then there exist a complex number $\lambda$ such that

    \begin{center}
   \(\varphi(A)=\lambda A, \)
    \, \, \(\varphi(B_{ij})=\lambda B_{ij}\),
    \, \,\(\varphi(C_i)=\lambda C_i\),

\vspace{2mm}

    \(\psi(A)=\lambda A\),
    \, \,\(\psi(B_{ij})=\lambda B_{ij}\),
    \, \,\(\psi(C_i)=\lambda C_i\).
        
\end{center}
\end{prop}

\begin{proof}
Since $\varphi$ and $\psi$ have degree $0$, we have

\begin{align}
    \varphi(A)&=\alpha A + \sum_{\substack{i\neq j \\ i,j=1}}^n  \alpha_{ij} B_{ij} \label{eq:varA}\\[1em]
    \varphi(B_{ij})&=\beta_{ij} A + \sum_{\substack{k \neq l \\ k,l=1}}^n \beta_{kl}^{ij} B_{kl}\label{eq:varB}\\[1em]
    \varphi(C_i)&=\sum_{j=1}^{n} \epsilon_{ij}C_j\label{eq:varC}\\[1em]
    \psi(A)&=\lambda A + \sum_{\substack{i\neq j \\ i,j=1}}^n \lambda_{ij} B_{ij}\label{eq:psiA}\\[1em]
    \psi(B_{ij})&=\mu_{ij} A + \sum_{\substack{k \neq l \\ k,l=1}}^n \mu_{kl}^{ij} B_{kl}\label{eq:psiB}\\[1em]
    \psi(C_i)&=\sum_{j=1}^n \sigma_{ij}C_j\label{eq:psiC},
\end{align}
where $\alpha,\alpha_{ij},\beta_{ij},\beta_{kl}^{ij},\epsilon_{ij},\lambda,\lambda_{ij},\mu_{ij},\mu_{kl}^{ij},\sigma_{ij}\in\mathbb{C}$.

Applying \Cref{eq:complex_num_homo0} to the elements $B_{ij}$ and $C_r$, i.e. $[\varphi(B_{ij}),C_r]=[B_{ij},\psi(C_r)]$, we therefore obtain
\begin{equation*}
    \beta_{ij} C_r + \sum_{\substack{k\neq l \\ k,l=1}}^n \beta_{kl}^{ij} [B_{kl},C_r]= \sum_{t=1}^n \sigma_{rt} [B_{ij},C_t].
\end{equation*}
Given that the brackets involved arise from the supercommutators of matrices, we obtain

\begin{equation}\label{eq:complex_num_homo0_1}
    \beta_{ij} C_r + \sum_{\substack{k \neq r \\ k=1}}^n \beta_{kr}^{ij} C_k - \sum_{\substack{l \neq r \\ l=1}}^n \beta_{rl}^{ij} C_l =
\sigma_{rj}C_i - \sigma_{ri}C_j.
\end{equation}

Since $n\geq3$, there exists \(r \neq i,j\) and, by the linear independence of the matrices $\left\{C_1,\ldots,C_n\right\}$ and the last equation, it follows that $\beta_{ij}=0$ and \Cref{eq:complex_num_homo0_1} became
\begin{equation}\label{eq:complex_num_homo0_2}
   \sum_{\substack{k \neq r \\ k=1}}^n \beta_{kr}^{ij} C_k - \sum_{\substack{l \neq r \\ l=1}}^n \beta_{rl}^{ij} C_l =
\sigma_{rj}C_i - \sigma_{ri}C_j.
\end{equation}
 If we now consider $r=i$, by \Cref{eq:complex_num_homo0_2} it follows $\sigma_{ij}=0$, for every $1\leq i,j\leq n$, with $i\neq j$.
 If instead we consider \(r \neq i,j\) and \(s \neq r\), \Cref{eq:complex_num_homo0_2} gives the condition became

\begin{equation}\label{eq:complex_num_homo0_3}
   \sum_{\substack{k \neq r \\ k=1}}^n \beta_{kr}^{ij} C_k - \sum_{\substack{l \neq r \\ l=1}}^n \beta_{rl}^{ij} C_l = 0_{n+1},
\end{equation}
 and then $\beta_{sr}^{ij}=\beta_{rs}^{ij}$. So, since \(B_{ij}=-B_{ji}\), then \(\varphi(B_{ij})=\sum_{k<l} (\beta_{kl}^{ij}-\beta_{lk}^{ij})B_{kl}=(\beta_{ij}^{ij}-\beta_{ji}^{ij})B_{ij}\). 
Applying \Cref{eq:complex_num_homo0} to the elements $C_r$ and $B_{ij}$, i.e. $[\varphi(C_r),B_{ij}]=[C_r,\psi(B_{ij})]$, with similar computations, we rewrite \Cref{eq:varB,eq:varC,eq:psiB,eq:psiC}, which are

\begin{align}
    \varphi(B_{ij})&=\tilde{\beta}_{ij} B_{ij} \label{eq:varB1}\\
    \varphi(C_i)&=\tilde{\epsilon}_iC_i\label{eq:varC1}\\
    \psi(B_{ij})&=\tilde{\mu}_{ij}B_{ij}\label{eq:psiB1}\\
    \psi(C_i)&=\tilde{\sigma}_i C_i\label{eq:psiC1},
\end{align}
with $\tilde{\beta}_{ij}=\beta_{ij}^{ij}-\beta_{ji}^{ij}$, $\tilde{\epsilon}_i=\epsilon_{ii}$, $\tilde{\mu}_{ij}=\mu_{ij}^{ij}-\mu_{ji}^{ij}$, and $\tilde{\sigma}_i=\sigma_{ii}$.

Now, applying \Cref{eq:complex_num_homo0} to \(A\) and \(C_r\), i.e. \([\varphi(A),C_r]=[A,\psi(C_r)]\), we get

\begin{equation*}
  \alpha C_r + \sum_{\substack{k \neq l \\ k,l=1}}^n \alpha_{kl} [B_{kl},C_r]= \tilde{\sigma}_r C_r.
\end{equation*}

\noindent Again, the brackets are given by the supercommutator of matrices, and therefore we obtain

\begin{equation*}
  \alpha C_r + \sum_{\substack{k \neq r \\ k=1}}^n \alpha_{kr} C_k -\sum_{\substack{l \neq r \\ l=1}}^n \alpha_{rl} C_l= \tilde{\sigma}_r C_r
\end{equation*}

Since the matrices $\left\{C_1,\ldots,C_n\right\}$ are linearly independent, we obtain \(\tilde{\sigma}_r=\alpha\) and \(\alpha_{sr}=\alpha_{rs}\) for every \(r,s\in\left\{1,\ldots,n\right\}\), with $r\neq s$. Since \(B_{ij}=-B_{ji}\), we get
\(\varphi(A)=\alpha A\). Similary, if we conceirn to \(C_r\) and \(A\) and we apply \Cref{eq:complex_num_homo0}  i.e. \([\varphi(C_r),A]=[C_r,\psi(A)]\), we obtain \(\tilde{\epsilon}_i=\lambda\) and \(\lambda_{ij}=\lambda_{ji}\) for every \(i,j\in\left\{1,\ldots,n\right\}\).
Therefore we rewrite \Cref{eq:varA,eq:psiA} and \Cref{eq:varB1,eq:varC1,eq:psiB1,eq:psiC1}:

\begin{align*}
   \varphi(A)&=\alpha A \\
    \varphi(B_{ij})&=\tilde{\beta}_{ij} B_{ij}\\
    \varphi(C_i)&=\lambda C_i\\
    \psi(A)&=\lambda A\\
    \psi(B_{ij})&=\tilde{\mu}_{ij}B_{ij}\\
    \psi(C_i)&=\alpha C_i.
\end{align*}
   
\vspace{2mm}

By considering the brackets \([\varphi(B_{ij}),C_i]=[B_{ij},\psi(C_i)]\) we obtain \(\tilde{\beta}_{ij}=\alpha\), and by 
\([\varphi(C_i),B_{ij}]=[C_i,\psi(B_{ij})]\) we obatin
\(\tilde{\mu}_{ij}=\lambda\). Lastly, from the brackets \([\varphi(B_{ij}),B_{jl}]=[B_{ij},\psi(B_{jl})]\)  we have \(\alpha=\lambda\) and the proof is complete.

\end{proof}

\begin{prop}\label{prop:VarPhigrade1}
Let $n\geq3$ and let \(\varphi, \psi\colon \Der(L)\to\Der(L)\) be two homogeneous maps of degree $1$ that satisfies \Cref{eq:complex_num_homo0}. Then \(\varphi=\psi=0\).

\end{prop}

\begin{proof}
Since $\varphi$ and $\psi$ have degree $1$, we have

\begin{align}
    \varphi(A)&=\sum_{k=1}^n\alpha_{k} C_{k}\label{eq:varA2}\\[1em]
    \varphi(B_{ij})&=\sum_{k=1}^n \beta_{k}^{ij} C_{k}\label{eq:varB2}\\[1em]
    \varphi(C_i)&=\delta_i A + \sum_{\substack{k \neq l \\ k,l=1}}^n \epsilon_{kl}^i B_{kl}  \label{eq:varC2}\\[1em]
    \psi(A)&=\sum_{k=1}^n\gamma_{k} C_{k} \label{eq:psiA2}\\[1em]
    \psi(B_{ij})&=\sum_{k=1}^n \lambda_{k}^{ij} C_{k} \label{eq:psiB2}\\[1em]
    \psi(C_i)&=\sigma_i A + \sum_{\substack{k \neq l \\ k,l=1}}^n \tau_{kl}^i B_{kl} \label{eq:psiC2}
\end{align}
where \(\alpha_k,\beta_k^{ij},\delta_i,\epsilon_{kl}^i,\gamma_k,\lambda_k^{ij},\sigma_i,\tau_{kl}^i \in \mathbb{C}\).
If we apply \Cref{eq:complex_num_homo0} to \(A\), i.e. \([\varphi(A),A]=[A,\psi(A)]\), we obtain the following identity

\begin{equation}\label{eq:rel_alpha-gamma}
    -\sum_{k=1}^n \alpha_k C_k = \sum_{k=1}^n \gamma_k C_k,
\end{equation}
that is to say \(\alpha_k=-\gamma_k\) for every \(k\).  Instead, if we apply \Cref{eq:complex_num_homo0} to \(A\) and \(B_{ij}\), i.e. \([\varphi(A),B_{ij}]=[A,\psi(B_{ij})]\), we obtain the identity

\begin{equation*}
    \sum_{k=1}^n \alpha_k [C_k,B_{ij}]= \sum_{k=1}^n \lambda_k^{ij} [A,C_k]
\end{equation*}
Therefore, we have:

\begin{equation}
    -\alpha_j C_i + \alpha_i C_j = \sum_{k=1}^n \lambda_k^{ij} C_k.
\end{equation}
Hence \(\lambda_k^{ij}=0\) for every \(k \neq i,j\), \(\lambda_i^{ij}=-\alpha_j\), and \(\lambda_j^{ij}=\alpha_i\), for every \(i,j\in\left\{1,\ldots,n\right\}\).

With similar arguments, if we apply \Cref{eq:complex_num_homo0} to \(B_{ij}\) and \(A\), i.e. \([\varphi(B_{ij}),A]=[B_{ij},\psi(A)]\), we obtain \(\beta_k^{ij}=0\), for every \(k \neq i,j\), and \(\beta_i^{ij}=-\gamma_j\), \(\beta_j^{ij}=\gamma_i\), for every \(i,j\in\left\{1,\ldots,n\right\}\). By \Cref{eq:rel_alpha-gamma} we have 
\(\beta_i^{ij}=-\gamma_j=\alpha_j\) and
\(\beta_j^{ij}=\gamma_i=-\alpha_i\). Hence we rewrite \Cref{eq:varB2,eq:psiA2,eq:psiB2}, wich are

\begin{align}
    \varphi(B_{ij})&=\alpha_j C_i -\alpha_i C_j \label{eq:varB3}\\
    \psi(A)&=-\sum_{k=1}^n \alpha_kC_k\label{eq:psiA3}\\
    \psi(B_{ij})&=-\alpha_jC_i +\alpha_iC_j.\label{eq:psiB3}
\end{align}

Since \(n \geq 3\), there exists \(l\neq i,j\). Then, if we apply \Cref{eq:complex_num_homo0} to \(B_{ij},B_{il}\), i.e. \([\varphi(B_{ij}),B_{il}]=[B_{ij},\psi(B_{il})]\), we obtain

\begin{equation*}
    [\alpha_jC_i-\alpha_iC_j,B_{il}]=[B_{ij},-\alpha_lC_i-\alpha_iC_l],
\end{equation*}
that is to say $\alpha_j C_l=\alpha_l C_j$. Hence \(\alpha_j=0\), for every \(1\leq j\leq n\).

We can rewrite \Cref{eq:varA2,eq:varC2,eq:psiC2} and \Cref{eq:varB3,eq:psiA3,eq:psiB3}, wich are:

\begin{align*}
    \varphi(A)&=\varphi(B_{ij})=0\\[0.5em]
    \varphi(C_i)&=\delta_i A + \sum_{\substack{k \neq l\\ k,l=1}}^n \epsilon_{kl}^i B_{kl}\\[0.5em]
    \psi(A)&=\psi(B_{ij})=0\\[0.5em]
    \psi(C_i)&=\sigma_i A + \sum_{\substack{k \neq l\\ k,l=1}}^n \tau_{kl}^i B_{kl}
\end{align*}

By applying \Cref{eq:complex_num_homo0} to \(B_{ij}\) and \(C_k\), we have

\begin{align*}
    0&=[\varphi(B_{ij}),C_k]=[B_{ij},\psi(C_k)]\\
    &=\sum_{\substack{r \neq s\\ r,s=1}}^n \tau_{rs}^k [B_{ij},B_{rs}]\\
    &=\sum_{\substack{r \neq s\\ r,s=1}}^n \tau_{rs}^k \big\{ 
    \delta_{jr} B_{is} +
    \delta_{is} B_{jr} +
    \delta_{js} B_{ri} +
    \delta_{ri} B_{sj} \big\}\\
    &=\sum_{\substack{s \neq j\\ s=1}}^n \tau_{js}^k B_{is}+
    \sum_{\substack{r \neq i\\ r=1}}^n \tau_{ri}^k B_{jr}+
    \sum_{\substack{r \neq j\\ r=1}}^n \tau_{rj}^k B_{ri}+
    \sum_{\substack{s \neq i\\ s=1}}^n \tau_{is}^k B_{sj}\\
    &=\sum_{\substack{s \neq j\\ s=1}}^n (\tau_{js}^k - \tau_{sj}^k) B_{is}+
    \sum_{\substack{s \neq i\\ s=1}}^n (\tau_{si}^k-\tau_{is}^k) B_{js},
\end{align*}
because \(B_{ij}=-B_{ji}\). Since \(i \neq j\), the elements $\left\{B_{is},B_{js}\right\}$ are linearly independent, hence \(\tau_{rs}^t=\tau_{sr}^t\), for every \(r,s,t\in\left\{1,\ldots,n\right\}\). Thus we have

\begin{equation*}
  \psi(C_k)=\sigma_k A+\sum_{\substack{k\neq l\\ k,l=1}}^n \tau_{kl}^i B_{kl}=\sigma_k A+\sum_{\substack{k < l \\ k,l=1}}^n (\tau_{kl}^i-\tau_{lk}^i) B_{kl}=\sigma_k A .
\end{equation*}

With similar computations on \([\varphi(C_k),B_{ij}]=[C_k,\psi(B_{ij})]\), one can obtain
\(\varphi(C_k)=\delta_k A\). Lastly, if we consider the brackets
\([\varphi(C_i),C_j]=[C_i,\psi(C_j)]\), with \(i \neq j\), we get 
\(\delta_i C_j= - \sigma_j C_i\). Then \(\delta_r=\sigma_s=0\), for every \(r,s\in\left\{1,\ldots,n\right\}\). 

\end{proof}

By \Cref{thm:Teorema_complete}, \Cref{prop:VarPhigrade0}, and 
\Cref{prop:VarPhigrade1}, the following statement is proved.

\begin{thm}\label{th:biderivation_heisenberg}
Let \(F\) a superbiderivation of $\Der(L)$, with \(n\geq3\). Then \(F\) has degree $0$. Moreover, there exist \(\lambda\) complex numbers such that

\begin{gather*}
    F(A,A)=F(A,B_{ij})=F(B_{ij},A)=F(C_i,C_j)=0,
    \\F(B_{ij},B_{kl})=\lambda[B_{ij},B_{kl}] \\F(A,C_i)=\lambda [A,C_i],\quad F(B_{ij},C_k)=\lambda [B_{ij},C_k]\\
    F(C_i,A)=\mu [C_i,A], \quad F(C_i,B_{jk})=\mu [C_i,B_{jk}].
\end{gather*}
\end{thm}

\noindent i.e. \(F(x,y)=\lambda[x,y]\) for every \(x,y \in L\).

\vspace{3mm}

We are now ready to classify the linear supercommuting maps for the complete Lie superalgebra $\Der(L)$.

\begin{prop}
A homogeneous map \(f\) of $\Der(L)$ is linear supercommuting if and only if \(f\) is homogeneous of degree $0$ and \(f(x)=\lambda x\), for some \(\lambda\in\mathbb{C}\).
\end{prop}

\begin{proof} The \textquotedblleft only if \textquotedblright direction is trivial. Conversely, let \(f\) an homogeneous linear supercommuting map of \(\Der(L)\). Then the map \(F \colon L \times L \to L\), with \(F(x,y)=[f(x),y]\), is a superbiderivation of $\Der(L)$. Indeed we have

\begin{itemize}
    \item Since the bracket is linear and $f$ is linear, \(F\) is also linear.
    \item Since the bracket is a homogeneous bilinear map and $f$ is homogeneous, \(F\) is also homogeneous and its degree coincides with \(\vert f\vert\).
    \item Since $f$ is a linear supercommuting map of $\Der(L)$, hence $F(x,-)=[f(x),-]=L_{f(x)}$ and $F(-,x)=[f(-),x]=[-,f(x)]=R_{f(x)}$, for any $x\in\Der(L)$.
\end{itemize}
If \(\vert f\vert=1\), by \Cref{th:biderivation_heisenberg} \(F\) is zero, and so $0=F(x,y)=[f(x),y]$. The last one implies that $ f(x)=0$, for every \(x \in \Der(L)\) beacase the center of $\Der(L)$ is trivial. Otherwise, if \(\vert f\vert=0\), by \Cref{th:biderivation_heisenberg} there exixst \(\lambda\in\mathbb{C}\) such that \([f(x),y]=\lambda[x,y]\) for every \(x,y \in L\). Thus $[f(x) - \lambda x,y]=0$, for all \(x\in\Der(L)\). Since the center of $\Der(L)$ is trivial, it follows \(f(x)=\lambda x\).
\end{proof}

\begin{corollary}
Let \(f\) be a linear supercommuting map of $\Der(L)$. Then $f$ has the form \(f(x)=\lambda x\), for some $\lambda\in\mathbb{C}$ and for all $x\in\Der(L)$.
\end{corollary}

\begin{proof}
The map \(f\) can be decomposed in an even and odd part, respectively $f_0$ and $f_1$ (see \Cref{prop:supercomm_decomposition}). From the previous proposition, the even part $f_0$ must be a multiple of the identity map and the odd part $f_1$ must be zero. 
\end{proof}

\subsection{Deformations of a Lie superalgebra by a superbiderivation}

The concluding segment of this section, which is dedicated to applications, adopts a geometric perspective. In this section, we introduce the concept of \emph{deformation of a Lie superalgebra} through the use of a superbiderivation. In the introduction, the applications of deformation theory for both Lie algebra and Lie superalgebra were discussed.

Following the tratement given in \cite{Binegar1986}, let $L$ be a Lie superalgebra over a field $\mathbb{F}$ of characteristic zero with brackets $[\cdot,\cdot]$. Let $\lambda$ be a parameter in $\mathbb{F}^\ast$, let $\phi_i\in C^2(L,L)$ (i.e. the vector space of all $2$-cochain for $L$ into itself) and let $[\cdot,\cdot]_\lambda$ be a new bracket defined as

\begin{equation}\label{eq:def_bracket}
  [x,y]_\lambda=[x,y]+\sum_{i=1}^\infty \lambda^i\phi_i(x,y).
\end{equation}

It must be noted that this does not a priori define a Lie superbracket on $L$. Further considerations must be made regarding the $\phi_i$ (for more details, see reference \cite{Binegar1986}). In this section, we demonstrate the possibility of deforming the bracket of a Lie superalgebra through the application of a superbiderivation. Let $B$ be a super-skewsymmetric superbiderivation of $L$ with $\vert B\vert=0$. We define the bracket (linearly) \emph{deformed by $B$}, denoted by $[\cdot,\cdot]_B$, as

\begin{equation}
    [x,y]_B=[x,y]+\lambda B(x,y),
\end{equation} 

for every $x,y\in L$. 
Since $L$ is a Lie superalgebra and $B$ has degree $0$, we have $[L_\alpha,L_\beta]\subseteq L_{\alpha+\beta}$ and $B(L_\alpha,L_\beta)\subseteq L_{\alpha+\beta}$, and then $[L_\alpha,L_\beta]\subseteq L_{\alpha+\beta}$. Hence this new deformed bracket is compatible with the $\Z_2$ graduation of $L$. Moreover, by the super-skewsymmetry of $B$ we have
\begin{align*}
    [x,y]_B&= [x,y]+\lambda B(x,y) \\
    &= -(-1)^{\vert x\vert \vert y\vert}[y,x]-(-1)^{\vert x\vert \vert y\vert}\lambda B(y,x) \\
    &= -(-1)^{\vert x\vert \vert y\vert}\left([y,x]+\lambda B(y,x)\right) \\
    &= -(-1)^{\vert x\vert \vert y\vert}[y,x]_B.
\end{align*}

We denote by $L_B$ the pair $\left(L,[\cdot,\cdot]_B\right)$. We collect these results into the following.

\begin{thm}
    Let $L$ be a Lie superalgebra and $B$ a bilinear map defined on $L$. Then:
    \begin{enumerate}
        \item If $\vert B\vert=0$, then $L_B$ is a superalgebra. 
        \item If $B$ is super-skewsymmetric, then $L_B$ is a supercommutative algebra. 
    \end{enumerate}
\end{thm}

As an immediate result, if we consider the usual supercommutator (as in \Cref{eq:eq:supercomm_decomposition}), the bracket induced by it on $L_B$, denoted by $[\cdot,\cdot]_B$, defines a Lie superalgebra structure on $L$. Thus the following is inferred by the last theorem and this last observation.

\begin{corollary}
    The pair $\left(L_B,[\cdot,\cdot]_B\right)$ is a Lie superalgebra.
\end{corollary}

Given the assumptions that have been made thus far, a comparison of our definition of a deformation of the bracket with those found in the extant literature—such as those related to the cohomology of Lie superalgebra—and with the topology of (sub)varieties of Lie superalgebra is already possible (for instance, nilpotent ones, to name just one example). Furthermore, our findings concerning superbiderivations—most notably our novel definition—could serve as a foundational starting point for future geometric and algebraic developments.

\section*{Acknowledgments and Funding}

Alfonso Di Bartolo was supported by the University of Palermo (FFR2024, UNIPA BsD).  Gianmarco La Rosa was supported by "Sustainability Decision Framework (SDF)" Research Project --  CUP B79J23000540005 -- Grant Assignment Decree No. 5486 adopted on 2023-08-04. The first and third authors were also supported by the “National Group for Algebraic and Geometric Structures, and Their Applications” (GNSAGA — INdAM).

\section*{ORCID}

Alfonso Di Bartolo \orcidlink{0000-0001-5619-2644} \href{https://orcid.org/0000-0001-5619-2644}{0000-0001-5619-2644}\\
Francesco Di Fatta \orcidlink{0009-0008-6197-6126} \href{https://orcid.org/0009-0008-6197-6126}{0009-0008-6197-6126}\\
Gianmarco La Rosa \orcidlink{0000-0003-1047-5993} \href{https://orcid.org/0000-0003-1047-5993}{0000-0003-1047-5993}

\printbibliography
\end{document}